\documentclass[12pt, reqno]{amsart}
\usepackage{hyperref}

\usepackage{amssymb,amsmath,graphicx,amsthm}
\usepackage{enumerate,xcolor,bigints}
\usepackage[normalem]{ulem}

\usepackage{color}

\def\T{\text}
\newcommand{\Om}{\Omega}

\newcommand{\no}[1]{\|{#1}\|}

\def\R{{\mathbb{R}}}

\def\C{{\mathbb{C}}}

\def\di{\partial}

\numberwithin{equation}{section}

\def\T{\text}

\theoremstyle{plain}

\newtheorem{theorem}{Theorem}[section]

\newtheorem{corollary}[theorem]{Corollary}

\newtheorem{proposition}[theorem]{Proposition}

\theoremstyle{definition}

\newtheorem{definition}[theorem]{Definition}

\theoremstyle{remark}

\newtheorem{remark}[theorem]{Remark}

\newcommand{\les}{\lesssim}

\def\B{\mathcal B}

\def\BB{\mathbb B}

\usepackage{amsmath}
\usepackage{amsthm}
\usepackage{amsfonts}
\usepackage{graphics}

\usepackage{amssymb}
\usepackage{verbatim}
\usepackage{amscd}
\usepackage{enumerate}
\usepackage{graphicx}
\usepackage{siunitx}
\usepackage{color}
\numberwithin{equation}{section}

\usepackage{mathabx}
\usepackage{mathtools}
\usepackage[tableposition=top]{caption}
\usepackage{booktabs,dcolumn}

\begin{document}

\title[Bergman-Toeplitz operators]{Bergman-Toeplitz operators on weakly pseudoconvex domains}



\email{tkhanh@uow.edu.au \, jiakunl@uow.edu.au \, ttp754@uowmail.edu.au }
  
\author[Tran Vu Khanh, Jiakun Liu, Phung Trong Thuc]{Tran Vu Khanh, Jiakun Liu, Phung Trong Thuc}

\address
	{Institute for Mathematics and its Applications, School of Mathematics and Applied Statistics,
	University of Wollongong,
	Wollongong, NSW 2522, AUSTRALIA.}	
 \thanks{Khanh was supported by ARC grant
 	DE160100173; Liu was supported by ARC grant DP170100929; Thuc was supported by PhD scholarship in ARC grant DE140101366}	
\begin{abstract}
We prove that for certain classes of pseudoconvex domains of finite type, the Bergman-Toeplitz operator $T_{\psi}$ with symbol $\psi=K^{-\alpha}$ maps from $L^{p}$ to $L^{q}$ continuously with $1< p\le q<\infty$ if and only if $\alpha\ge\frac{1}{p}-\frac{1}{q}$,
where $K$ is the Bergman kernel on diagonal. 
This work generalises the results on strongly pseudoconvex domains by \v{C}u\v{c}kovi\'{c} and McNeal, and Abeta, Raissy and Saracco.  
\end{abstract}

\maketitle         
\section{Introduction} \label{S1}
 Let $\Omega$ be a bounded domain in $\mathbb C^n$ with the boundary $b\Om$.   A fundamental object associated to $\Om$ is  the Bergman projection $P$, that is the orthogonal projection of $L^2(\Om)$ onto the closed subspace of square-integrable holomorphic functions on $\Om$. The Bergman projection can be expressed via the integral representation 
$$Pu(z)=\int_\Om K(z,w)u(w)dw,$$
where $dw$ is the Lebesgue measure on $\Om$ and the integral kernel $K$ is called the Bergman kernel. It is well known that the Bergman projection maps from $L^p(\Om)$ to itself with $1<p<\infty$ on some classes of pseudoconvex domains of finite type such as strongly pseudoconvex domains \cite{PhSt77}, convex domains of finite type \cite{McNSt94}, pseudoconvex domains of finite type in $\C^2$ \cite{NaRoStWa89} (see also \cite{McN94b, KhRa15}).\\
 
Let $\psi\in L^\infty(\Om)$, the Bergman-Toeplitz operator with symbol $\psi$ is defined by 
$$T_\psi f (z)=P(\psi f) (z) = \int_\Om K(z,w)\psi(w)f(w) dw.$$  
The study of the Bergman-Toeplitz operators has become a central topic since it is at the interface of many important fields in algebra and analysis, e.g., $C^*$-algebra, operator theory, harmonic analysis, pseudodifferential operators, and several complex variables (see \cite{BoAl06,SSU89,Upm96} and references therein).   In this work, we focus on  the ``gain" $L^p$-estimate property of these operators by the effect of their symbols.
It is clear that $T_\psi: L^p(\Om)\to L^p(\Om)$ continuously if $P:L^p(\Om)\to L^p(\Om)$  continuously, as $\psi\in L^\infty(\Om)$. 
In order to improve the regularity of $T_\psi$ in $L^p$ spaces, one should compensate  by choosing $\psi$ such as $\psi(z)\to 0$ as  $z\to b\Om$.
Working on strongly pseudoconvex domains, \v{C}u\v{c}kovi\'{c} and McNeal \cite{CuMc06} study this gain property by choosing $\psi=\delta^\eta$ with $\eta>0$, where $\delta(\cdot)=d(\cdot,\partial\Om)$ is the Euclidean distance from the boundary.  Precisely, they prove the following result: 
\begin{theorem}[\v{C}u\v{c}kovi\'{c} and McNeal \cite{CuMc06}]\label{thm:CM} Let $\Omega$ be a smooth, bounded, strongly pseudoconvex domain in $\C^n$ and let $\eta\ge 0$. 
	\begin{enumerate}
		\item If $0\le \eta < n+1$ and $1<p<\infty$, then 
		\begin{enumerate}
			\item[(i)] if  $\frac{n+1}{n+1-\eta}< \frac{p}{p-1}$, then $T_{\delta^\eta}: L^p(\Om)\to L^q(\Om)$ continuously, where $\frac{1}{q}=\frac{1}{p}-\frac{\eta}{n+1}$;
			\item[(ii)] if $\frac{n+1}{n+1-\eta}\ge \frac{p}{p-1}$, then $T_{\delta^\eta}: L^p(\Om)\to L^q(\Om)$ continuously, for all $p\le q<\infty$.
	\end{enumerate}			
		\item If $\eta\ge n+1$, then $T_{\delta^\eta}: L^1(\Om)\to L^\infty(\Om)$ continuously.
		\end{enumerate}
	\end{theorem}
	
Later on, Abate, Raissy and Saracco \cite{ARS12} show that the gain exponents  in Part (1.i) of Theorem~\ref{thm:CM} are also optimal by using geometric characterisation of Carleson measures in term of the intrinsic Kobayashi geometry. In fact, they prove:
\begin{theorem}[Abate, Raissy and Saracco \cite{ARS12}]\label{thm:ARS}
	Let $\Om$ be a smooth, bounded, strongly pseudoconvex domain in $\C^n$ and let $1<p<q<\infty$ and $\eta\ge0$. Then $T_{\delta^{\eta}}:L^p(\Om)\to L^q(\Om)$ continuously if and only if $\frac{\eta}{n+1}\ge\frac{1}{p}-\frac{1}{q}$.
\end{theorem}

A crucial ingredient in the proof of Theorem~\ref{thm:CM} is the precise information on the Bergman kernel established by Fefferman \cite{Fef74} on strongly pseudoconvex domains. Although, the authors in \cite{CuMc06} commented {\it``our proof of Theorem~\ref{thm:CM} goes through, with minimal changes, on other classes of domains where good estimates on the Bergman kernel are known, e.g., finite type domains in $\C^2$, convex domains of finite type in $\C^n$''}, we observe that for weakly pseudoconvex domains, the ``good estimates'' on Bergman kernels (if established) depend on the multi-type of boundary points, then the symbol should also depend on this multi-type (the multi-type is uniform on strongly pseudocovex domains). Thus the symbol $\delta^\eta$ with a fixed $\eta$ is not a suitable candidate for the study of gain $L^p$ estimates of Toeplitz operators on weakly pseudoconvex domains. For this reason, we shall work on the symbol $K^{-\alpha}(z,z)$ for some constant $\alpha>0$ instead of the symbol $\delta^\eta(z)$, and then generalise \v{C}u\v{c}kovi\'{c} and McNeal's result in \cite{CuMc06} and Abate, Raissy and Saracco's result in \cite{ARS12} for a large class of pseudocovex domains of finite type (see Theorem \ref{thm:main S1}). 
In particular, Bergman kernels in this class have good estimates, called ``sharp $\B$-type".  Additionally, we also provide an upper-bound for the norm $\no{T_\psi}_{L^p(\Om)\to L^q(\Om)}$. The upper-bound for the norm of Bergman projection has been given by  Zhu \cite{Zhu06} on the unit ball $\mathbb{B}$ in $\C^n$, that is 
$$\no{P}_{L^p(\mathbb{B})\to L^p(\mathbb{B})}\le C\frac{p^2}{p-1}, \quad \T{for }1<p<\infty.$$ 
Recently, \v{C}u\v{c}kovi\'{c}  \cite{Cuc17} obtains this upper-bound for strongly pseudoconvex domains for a different constant $C$. \\

Our main result of this paper is:
 \begin{theorem}\label{thm:main S1}
 	Let $\Om$ be a bounded, pseudoconvex domain in $\C^n$ with smooth boundary. Assume that $\Om$ satisfies one of the following conditions:
 	\begin{enumerate}
 		\item[(a)] $\Om$ is a strongly pseudoconvex domain;
 		\item[(b)] $\Om$ is a pseudoconvex domain of finite type and $n=2$;
 		\item[(c)] $\Om$ is a convex domain of finite type;
 		\item[(d)] $\Om$ is a decoupled domain of finite type;
 		\item[(e)] $\Om$ is a pseudoconvex domain of finite type whose Levi-form has only one degenerate eigenvalue or comparable eigenvalues.  
 	\end{enumerate} 
 	Then we have the following conclusions:
 	\begin{enumerate}
 		\item The Bergman-Toeplitz operator $T_\psi$ with symbol $\psi(z)= K^{-\alpha}(z,z) $ maps from $L^p(\Om)$ to $L^q(\Om)$ continuously with $1<p\le q<\infty$ if and only if $\alpha\ge \frac{1}{p}-\frac{1}{q}$. Additional, if $T_\psi:L^p(\Om)\to L^q(\Om)$ continuously, then 
 			$$\no{T_\psi}_{L^p(\Om)\to L^q(\Om)}\le C_\Om \left(\frac{p}{p-1}+q \right)^{1-\frac{1}{p}+\frac{1}{q}},$$
 			where the constant $C_{\Om}$ depends only on the domain $\Om$.
 		\item The Bergman-Toeplitz operator $T_\psi$ with symbol $\psi(z)=K^{-1}(z,z)$ maps from $L^1(\Om)$ to $L^\infty(\Om)$ continuously.
 	\end{enumerate} 
 \end{theorem}

When $\Om$ is strongly pseudoconvex, from the asymptotic estimates of the Bergman kernel by Fefferman \cite{Fef74},
one has
 \begin{equation}\label{eqn:compare r and K}
 C_1|r(z)|^{-(n+1)}\le K(z,z)\le C_2 |r(z)|^{-(n+1)}\quad \T{ for all $z\in \Om$,}
 \end{equation}
 where $r$ is a defining function of $\Om$ satisfying $c_1\delta(z)\leq |r(z)| \leq c_2\delta(z)$, $\forall z\in\Om$, for two constants $c_1, c_2$. 
Using \eqref{eqn:compare r and K}, one can see that for strongly pseudoconvex domains, the conclusions for case $(a)$ in Theorem~\ref{thm:main S1} are equivalent to those in Theorems \ref{thm:CM} and \ref{thm:ARS}.\\

The proof of Theorem~~\ref{thm:main S1}  is a consequence of the results of Theorems~\ref{thm:main1}, \ref{thm:main2} and \ref{thm:main3} below. In Theorems~\ref{thm:main1} and \ref{thm:main2}, we give the statement in an abstract setting of domains. 
Our work may extend to other class of pseudoconvex domains, e.g., the $h$-extendible domains \cite{Yu94}. 
However, to the best of our knowledge, there are only domains listed in Theorem~~\ref{thm:main S1} whose Bergman kernels have the desired ``good estimates".  
In fact, we prove in  Theorem~\ref{thm:main1} that if the Bergman kernel satisfies the good estimate (named {\it sharp $\B$-type}), then the Bergman-Toeplitz operator $T_{\psi}$ with a symbol $\psi(z)\le  K^{\frac{1}{q}-\frac{1}{p}}(z,z) $ (almost everywhere)  for $1<p\le q<\infty$, maps from $L^p(\Om)$ to  $L^q(\Om)$ continuously.  
By an additional geometric hypothesis, in Theorem~\ref{thm:main2} we prove conversely that if $T_{\psi}:L^p(\Om)\to L^q(\Om)$ continuously with $\psi(z)=  K^{-\alpha}(z,z) $, then $\alpha \geq \frac{1}{p}-\frac{1}{q}$. 
Finally, we verify that all domains listed in Theorem~\ref{thm:main S1} satisfy the hypotheses of both Theorems~\ref{thm:main1} and \ref{thm:main2} by using the work of Fefferman \cite{Fef74}, Catlin \cite{Cat89}, McNeal \cite{McN94, McN91}, McNeal and Stein \cite{McNSt94}, and Cho \cite{Cho96,Cho02}.\\

 {\bf Notations:} Throughout this paper, we use letters $C$ and $\tilde C$ to denote universal positive constants that only depend on the domain $\Om$ (e.g. $n$ and the type of $\Om$) and $\no{\psi}_{L^\infty(\Om)}$, but may change from place to place.
 We also denote by $r_\Om(z)$ and $K_{\Om}$ the negative distance function and  the Bergman kernel associated to $\Om$, respectively.

\vspace{10pt}

\section{$L^p\T-L^q$ boundedness of Bergman-Toeplitz operators} \label{S2}
In this section, we  introduce a concept ``sharp $\B$-type''. 
Heuristically, if a Bergman kernel is of this type,  then it has good estimates that ensure the Bergman projection is self bounded in $L^p$. 
This sharp $\B$-type condition unifies all good estimates established by many authors on strongly pseudoconvex domains, pseudoconvex domains of finite type in $\C^2$, convex domains of finite type, and etc. \\

Let $\Omega'$ be a bounded domain in $\C^n$.  For $z'\in \overline{\Om'}$ near the boundary $b\Om'$, a family of functions $\B=\{b_j(z',\cdot)\}_{j=1}^n$  is called a \emph{$\B$-system} at $z'$, if there exists a neighbourhood $U$ of $z'$, a positive integer $m\geq2$ such that
$\forall$ $w'\in U$, 
  	$$
  	b_{1}(z',w')  := \frac{1}{\delta(z',w')} \quad \T{and} \quad 
  	b_{j}(z',w')   :=  \sum_{k=2}^{m}\left(\dfrac{A_{jk}\left(z'\right)}{\delta(z',w')}\right)^{\frac{1}{k}}, \quad \T{for } j=2,{\scriptstyle\ldots},n,$$
  	where $\left\{ A_{jk}:U\rightarrow\mathbb{R}^{\geq0}\right\}$ are bounded functions, and $\delta(z',w')$ is the pseudo-distance between $z'$ and $w'$, given by
  	\begin{eqnarray}
  	\label{eqn:delta0}
  	\delta(z',w') = \left|r_{\Om'}\left(z'\right)\right|+\left|r_{\Om'}\left(w'\right)\right|+\left|z'_{1}-w'_{1}\right|+\sum_{l=2}^{n}\sum_{s=2}^{m}A_{ls}\left(z'\right)\left|z'_{l}-w'_{l}\right|^{s},
  	\end{eqnarray}
  	under a proper system of coordinates, see \cite{McNSt94}.\\
	
Let us start with the definition of  sharp $\B$-type at a point near the boundary.	
  \begin{definition}\label{def:local B type}
  		The Bergman kernel $K_{\Om'}$  is said to be of {\it  $\B$-type at $z'\in \overline{\Om'}$ near the boundary $b\Om'$} if there exists a constant $c>0$ such that  $\forall$  $w'\in \Omega'\cap \BB(z',c)$,
  		$$\left|K_{\Om'}(z',w')\right|\le C\prod_{j=1}^{n}b^{2}_{j}(z',w'). $$
  		We also say that $K_{\Om'}$ is of {\it sharp $\B$-type at $z'\in \overline{\Om'}$ near the boundary $b\Om'$} if $K_{\Om'}$ is of $\B$-type and has the sharp lower-bound on diagonal, i.e.,
  		$$ \tilde{C} \prod_{j=1}^n b^2_j(z',z')\le K_{\Om'}(z',z')\le C \prod_{j=1}^{n}b^2_{j}(z',z').$$
  \end{definition}

Next we give the definition of (global) sharp $\B$-type on a domain $\Om$.
\begin{definition}\label{def:global B type}
	We  say that a kernel $K_\Om$ associated to a domain $\Om$ is of  sharp $\B$-type if
	\begin{enumerate}
		\item[(i)]  $K_\Om\in C(\overline\Om\times\overline\Om\setminus{\Delta_b})$ where $\Delta_b=\{(z,z):z\in b\Om\}$; and 
		\item[(ii)] $\forall$ $z\in \overline\Om$ near the boundary $b\Om$, there exists a biholomorphism $\Phi_z$ whose holomorphic Jacobian is uniformly nonsingular so that the Bergman kernel $K_{\Om'}$ associated to the domain $\Om':=\Phi_z(\Om)$ is of sharp $\B$-type at $z':=\Phi_z(z)$.
		\end{enumerate}
\end{definition} 
 
 \begin{remark}
 	The $\B$-type kernel condition is originally introduced by McNeal and Stein \cite{McNSt94}. In \cite{McNSt94}, they prove that if the Bergman kernel associated to a convex domain of finite type is of $\B$-type, then the Bergman projection is $L^p$ self-bounded for $1<p<\infty$.
 \end{remark}

The sharp $\B$-type implies the following integral estimate, which will be used subsequently. 

\begin{proposition}\label{prop:main} Let $\Om$ be a domain in $\C^n$.
Assume that the Bergman kernel $K_{\Om}$ is of sharp $\mathcal{B}$-type. 
Then, for each $z_0\in b\Om$, there is a neighbourhood $U$ of $z_0$ such that  for any $a\ge 1$ and $-1<b<2a-2$, 
	\begin{equation}
	\label{eqn:Iab}
	\begin{split}
	I_{a,b}\left(z\right) &:= \intop_{\Omega\cap U}\left|K_{\Om}\left(z,w\right)\right|^{a}\left|r_{\Om}\left(w\right)\right|^{b}dw \\
		&\leq C\frac{2a-1}{(2a-2-b)(b+1)}K_{\Om}(z,z)^{a-1}\left|r_{\Om}\left(z\right)\right|^{b}
	\end{split}
	\end{equation}
	for every $z\in\Omega\cap U.$
\end{proposition}
\begin{remark}\label{rm:rm1}
To obtain \eqref{eqn:Iab} for $a=1$, it suffices to assume that $K_\Om$ is of $\mathcal{B}$-type, instead of sharp $\mathcal{B}$-type.
\end{remark}
\begin{proof} We choose  $U$ a small neighbourhood of $z_0$ such that $\Phi_z(U)\subset \BB(z',c)$ for any $z\in U$,  where  the ball $\BB(z',c)$ and the biholomorphism $\Phi_z$ are given in Definition~\ref{def:local B type} and~\ref{def:global B type}, respectively. By the invariant formula of the Bergman kernel
	\begin{eqnarray}\label{eqn:invariance}
	K_{\Om}(z,w)=\det J_\C\Phi_z(z)K_{\Phi_z(\Om)}(\Phi_z(z),\Phi_z(w))\overline{\det J_\C\Phi_z(w)},
	\end{eqnarray}
	where $C^{-1}\le |\det J_\C\Phi_z(w)|\le C$ for all $w\in U$, 
we have 
$$I_{a,b}(z)\le C\int_{\Om'\cap \BB(z',c)}|K_{\Om'}(z',w')|^{a}|r_{\Om'}(w')|^{b}dw',
$$
where $\Om'=\Phi_z(\Om)$, $z'=\Phi_z(z)$, $w'=\Phi_z(w)$ and $r_{\Om'}(w')=r_{\Om}(\Phi_z^{-1}(w'))$. Thus, in order to show \eqref{eqn:Iab}, it suffices to show that 
	\begin{equation}\label{eqn:Iab prime}
	\begin{split}
I'_{a,b}\left(z'\right) &:=\intop_{\Omega'\cap \BB(z',c)}\left|K_{\Om'}\left(z',w'\right)\right|^{a}\left|r_{\Om'}\left(w'\right)\right|^{b}dw' \\
		& \leq C\frac{2a-1}{(2a-2-b)(b+1)}K_{\Om'}(z',z')^{a-1}\left|r_{\Om'}\left(z'\right)\right|^{b},
	\end{split}
	\end{equation} 
in which $K_{\Om'}$ is a kernel of sharp $\B$-type at $z'$.
Here and in what follows in this proof, we omit the subscript $\Om'$ and the superscript {\it prime} for convenience. 

	  It is clear that $b_j(z,w)\le b_j(z,z)$  and $b_1(z,w)\le \dfrac{1}{|r(z)|+|r(w)|}$ and hence 
	\begin{eqnarray}
	\label{eqn:Kzw Kzz}|K(z,w)|\le CK(z,z){\left(\frac{|r(z)|}{|r(z)|+|r(w)|}\right)^{2}},\quad \T{	for all $w\in \Om\cap \BB(z,c)$}.
	\end{eqnarray}
Since $a\geq1$, from \eqref{eqn:Kzw Kzz} one has 
	\begin{equation}\label{eqn:Jab}
	I_{a,b}(z)\le {C|r(z)|^{2a-2}K^{a-1}(z,z)J_{a,b}(z)},
	\end{equation}
where
	\begin{equation*}
		 J_{a,b}(z)=\int_{\Om}\frac{|r(w)|^b}{(|r(z)|+|r(w)|)^{2a-2}}|K(z,w)|dw.
	\end{equation*}
	We shall use the polar coordinates in $z_{k}-w_{k}$ with $\rho_{k}:=\left|z_{k}-w_{k}\right|$ for $k=2,\ldots,n$ and the change of variables  $\rho_1:=-r(w)$, $\xi:=\left|z_1-w_1\right|$, to estimate the integral $J_{a,b}$. 
	First, we define the increasing sequence $\{M_j(z,\rho)\}_{j=1}^n$ by induction: $M_1(z,\rho)=|r(z)|+\rho_1$ and
	\begin{eqnarray}
	\label{eqn:Mj}
	M_j(z,\rho)=M_{j-1}(z,\rho)+\sum_{k=2}^{m}A_{jk}\left(z\right)\rho_j^{k}, \quad \T{for}\quad  j=2,\dots,n.
	\end{eqnarray}
	We remark that $M_j$ is independent of $\rho_l$ for $l=j+1,j+2,\dots,n$.
	Then, observe that  $b_j(z,w)\le m\rho_{j}^{-1}$ and also
	\begin{equation}
	b_j(z,w)=\sum_{k=2}^m\left(\frac{A_{jk}(z)}{\delta(z,w)}\right)^{\frac{1}{k}}\le \sum_{k=2}^m\left(\frac{A_{jk}(z)}{M_{j-1}(z,\rho)}\right)^{\frac{1}{k}}
	\end{equation}
	since $\delta(z,w)=\xi+M_n(z,\rho)$ is greater than both  $A_{jk}\rho_j^k$ and $M_{j-1}(z,\rho)$. Therefore, the integral $J_{a,b}$, after changing coordinates, can be estimated as
	\begin{equation*}
		\begin{split}
			J_{a,b}(z)\le& C\bigintsss_0^\infty \!\!\!\cdots\!\! \bigintsss_0^\infty\frac{\rho_1^b}{M_1(z,\rho)^{2a-2}}\prod_{j=2}^n\left[\sum_{k=2}^m\left(\frac{A_{jk}(z)}{M_{j-1}(z,\rho)}\right)^{\frac{1}{k}}\right]\frac{1}{\left(\xi+M_n(z,\rho)\right)^2}d\rho_1\cdots d\rho_nd\xi\\
			\le&C\bigintsss_0^\infty \!\!\!\cdots\!\! \bigintsss_0^\infty\frac{\rho_1^b}{M_1(z,\rho)^{2a-2}}\prod_{j=2}^n\left[\sum_{k=2}^m\left(\frac{A_{jk}(z)}{M_{j-1}(z,\rho)}\right)^{\frac{1}{k}}\right]\frac{1}{M_n(z,\rho)}d\rho_1\cdots d\rho_n,
		\end{split}
	\end{equation*}
	where the second inequality follows by integrating with respect to $\xi$. We compute this integral by the following claim. \\
	
	{\bf Claim:} {\it For $j=2,\dots,n$, we have 
		\begin{eqnarray}
		\label{eqn:int Mj}
		\left[\sum_{k=2}^m\left(\frac{A_{jk}(z)}{M_{j-1}(z,\rho)}\right)^{\frac{1}{k}}\right]\bigintssss_0^\infty\frac{d\rho_j}{M_j(z,\rho)}\le \frac{C}{M_{j-1}(z,\rho)}.
		\end{eqnarray}}
	\begin{proof}[Proof of Claim.] Since $M_j(z,\rho)\ge M_{j-1}(z,\rho)+A_{jk}(z)\rho_j^k$ for all $k=2,\dots,m$, the LHS of \eqref{eqn:int Mj} is bounded by
		\begin{eqnarray}
		\label{eqn:int Mj1}
		\sum_{k=2}^m	\left(\frac{A_{jk}(z)}{M_{j-1}(z,\rho)}\right)^{\frac{1}{k}}\bigintssss_0^\infty\frac{d\rho_j}{M_{j-1}(z,\rho)+A_{jk}(z)\rho_j^k}.
		\end{eqnarray}
		For each $k=2,\dots,m$, if $A_{jk}(z)=0$, there is nothing to do; otherwise, we use the change of the coordinate  $x:=A_{jk}^{\frac{1}{k}}(z)\rho_j$  to get 
		\begin{equation*}
		\begin{split}
			\left(\frac{A_{jk}(z)}{M_{j-1}(z,\rho)}\right)^{\frac{1}{k}}\bigintsss_0^\infty \frac{d\rho_j}{M_{j-1}(z,\rho)+A_{jk}(z)\rho_j^k} &=	\frac{1}{\left(M_{j-1}(z,\rho)\right)^{\frac{1}{k}}}\bigintsss_0^\infty \frac{d x}{M_{j-1}(z,\rho)+x^k} \\
				& \le \frac{C}{M_{j-1}(z,\rho)}.
		\end{split}
		\end{equation*}
	This proves the claim.
	\end{proof}
	Coming back to the computation of $J_{a,b}(z)$, by the claim \eqref{eqn:int Mj} and an induction argument from $j=n$ to $j=2$, we have
	\begin{equation}\label{eqn:Jab end}
	J_{a,b}(z)\le C\bigintsss_0^\infty \frac{\rho_1^bd\rho_1}{M_1(z,\rho)^{2a-1}} =C\bigintsss_0^\infty \frac{\rho_1^bd\rho_1}{(|r(z)|+\rho_1)^{2a-1}}=C|r(z)|^{b-2a+2}\kappa_{a,b},
	\end{equation}
	where $$\kappa_{a,b}=\int_{0}^\infty\frac{t^bdt}{(1+t)^{2a-1}}\le \int_0^1t^bdt+\int_{1}^{\infty}t^{b-2a+1}dt=\frac{2a-1}{(2a-2-b)(b+1)}.$$
	By \eqref{eqn:Jab} and \eqref{eqn:Jab end}, the conclusion of Proposition~\ref{prop:main} follows.
\end{proof}

The following corollary is a combination of a generalised Schur's test (Theorem~\ref{thm:Schur's test} below) and Proposition~\ref{prop:main}.
\begin{corollary}\label{co1} Let $\Om$ be a bounded domain in $\C^n$ and $1<p\le q<\infty$. 
Assume that the Bergman kernel $K_{\Om}$ is  of sharp $\B$-type and $\psi:\Om\to \C$ satisfies 
$$\left|\psi(z)\right|\le K_{\Om}^{\frac{1}{q}-\frac{1}{p}}(z,z)$$ 
almost everywhere. Then, for each $z_0\in b\Om$ there exists a neighbourhood $U$ of $z_0$ such that the Toeplitz operator $T_{\psi,U}$ defined by
$$(T_{\psi,U}u)(z):=\int_{\Om\cap U}K_{\Om}(z,w)\psi(w)u(w)dw\quad \T{for }z\in \Om\cap U, $$ 
maps from $L^p(\Om\cap U)$ to $L^q(\Om\cap U)$ continuously and 
\begin{eqnarray}
\label{eqn:norm T local}
\no{T_{\psi,U}u}_{L^q(\Om\cap U)}\le C\left(\frac{p}{p-1}+q\right)^{1-\frac{1}{p}+\frac{1}{q}}\no{u}_{L^p(\Om\cap U)}
\end{eqnarray}
for all $u\in L^p(\Om\cap U)$, where $C$ is independent of $p,q$ and $u$.
\end{corollary}
\begin{proof}
	Let $\alpha=\dfrac{1}{p'}=1-\dfrac1p$ and $0< \beta<\min\{\dfrac{1}{p'},\dfrac{1}{q}\}$. Apply Proposition~\ref{prop:main} twice for the pair $(a,b)$ replaced by $(1,-\beta p')=(\alpha p',-\beta p')$ and by $(\frac{q}{p},-\beta q)=((1-\alpha)q,-\beta q)$, we have
	\begin{eqnarray}
	I_{(1,-\beta p')}(z)&=&\int_{\Om\cap U}|K_{\Om}(z,w)|^{\alpha p'}\left(|r_{\Om}(w)|^{-\beta }\right)^{p'}dw\le C_1 \left(|r_{\Om}(z)|^{-\beta}\right)^{p'},\\
		I_{(\frac{q}{p},-\beta q)}(w)&=&\int_{\Om\cap U}|K_{\Om}(z,w)|^{(1-\alpha)q}\left(|r_{\Om}(z)|^{-\beta }\right)^{q}dz\le C_2 \left(K_{\Om}(w,w)^{\frac{(1-\alpha)q-1}{q}}|r_{\Om}(w)|^{-\beta}\right)^{q} \nonumber
	\end{eqnarray}
where 
$$C_1=\frac{C}{\beta p'(1-\beta p')}\quad \T{and }\quad C_2= C\frac{\frac{2q}{p}-1}{(\frac{2q}{p}-2+\beta q)(1-\beta q)}. $$
Notice that, by the hypothesis of $\psi$, $$(|r_{\Om}(w)|^{-\beta})^{-1} K_{\Om}(w,w)^{\frac{(1-\alpha)q-1}{q}} |r_{\Om}(w)|^{-\beta}\psi(w)=\psi(w) K_{\Om}(w,w)^{\frac{1}{p}-\frac{1}{q}}\le 1$$
for all $w\in\Om$. 
Thus, the hypothesis of Theorem~\ref{thm:Schur's test} holds for 
	\begin{eqnarray*}
		X=Y=\Om\cap U, & &  h_1(w)=|r_{\Om}(w)|^{-\beta}, \\
		h_2(w)=K_{\Om}(w,w)^{\frac{(1-\alpha)q-1}{q}}|r_{\Om}(w)|^{-\beta},~~&\T{ and }&~~ g(z)=|r_{\Om}(z)|^{-\beta}.
	\end{eqnarray*}
Therefore, $T_{\psi, U}:L^p(\Om\cap U)\to L^q(\Om\cap U)$ continuously and  
$$\no{T_{\psi, U}u}_{ L^q(\Om\cap U)}\le C\inf_{0< \beta<\min\{\frac{1}{p'},\frac{1}{q}\}} \{\tau(\beta)\}\no{u}_{L^p(\Om\cap U)},$$
where 
$$\tau(\beta)=\left(\frac{1}{\beta p'(1-\beta p')}\right)^{\frac{1}{p'}}\left(\frac{\left(\frac{2q}{p}-2\right)+1}{\left(\left(\frac{2q}{p}-2\right)+\beta q\right)(1-\beta q)}\right)^{\frac{1}{q}}.$$

We finish the proof of this corollary by showing that 
\begin{eqnarray}
\label{eqn:tau}\inf_{0< \beta<\min\{\frac{1}{p'},\frac{1}{q}\}} \tau(\beta)\le 4 \left(\frac{p}{p-1}+q\right)^{1-\frac{1}{p}+\frac{1}{q}}.
\end{eqnarray}
We first get rid of term $\left(\frac{2q}{p}-2\right)$ in $(\cdots)^{\frac{1}{q}}$ by the inequality $\frac{x+a}{x+b}\le \frac{a}{b}$ for $a\ge b>0$ and $x\ge 0$. Then we choose $\beta=\frac{1}{p'+q}\le \min\{\frac{1}{p'},\frac{1}{q}\}$. It follows, 
$$\beta p'(1-\beta p')=\beta q(1-\beta q)=\frac{(p'+p)^2}{p'q}=\left(\frac{1}{p'}+\frac{1}{p}\right)(p'+p).$$ Therefore,
$$\tau\left(\frac{1}{p'+q}\right)\le\left[ \left(\frac{1}{p'}+\frac{1}{p}\right)(p'+p)\right]^{\frac{1}{p'}+\frac{1}{p}}\le 4\left(p'+p\right)^{\frac{1}{p'}+\frac{1}{p}}=4\left(\frac{p}{p-1}+q\right)^{1-\frac{1}{p}+\frac{1}{q}},$$
where the second inequality follows by $x^x\le 4$ for $x=\frac{1}{p'}+\frac{1}{p}\in [0,2]$. This proves \eqref{eqn:tau}.
\end{proof}

The main result of this section is the following theorem, in which we prove the gain $L^p\T-L^q$ estimate of Bergman-Toeplitz operators $T_\psi$.

\begin{theorem}\label{thm:main1}
	Let $\Om$ be a bounded domain in $\C^n$. Assume that the Bergman kernel  $K_{\Om}$ is of sharp $\B$-type. 
	\begin{enumerate}
		\item 	If $\left|\psi(z)\right|\le \big(K_{\Om}(z,z)\big)^{\frac{1}{q}-\frac{1}{p}}$ almost everywhere with  $1<p\le q<\infty$, then $T_{\psi}:L^p(\Om)\to L^q(\Om)$ continuously. Furthermore, 
		$$\no{T_\psi}_{L^p(\Om)\to L^q(\Om)}\le C_{\Om,\psi} \left(\frac{p}{p-1}+q\right)^{1-\frac{1}{p}+\frac{1}{q}},$$
		where $C_{\Om,\psi}$ is independent of $p,q$.
		\item  If $\left|\psi(z)\right|\le \big(K_{\Om}(z,z)\big)^{-1}$ almost everywhere, then $T_{\psi}:L^1(\Om)\to L^\infty(\Om)$ continuously.
	\end{enumerate} 
	
\end{theorem}
\begin{proof}
We choose a partition of unity $\{\chi_j\}_{j=0}^{N}$ and a covering $\{U_j\}_{j=0}^N$ to $\overline\Om$ 
so that $ {\T{supp}(\chi_j)}\Subset U_j$, $U_{0}\Subset \Omega$,  $b\Omega\subset\bigcup_{j=1}^{N}U_{j}$,
and the integral estimates   in Proposition~\ref{prop:main} hold on $U_j$ for all $j=1,\dots, N$. Denote by ${\bf 1}_A$ the characteristic function of a $A\subset \Om$. So we can decompose $T_\psi u$ as 
$$T_{\psi}u=\sum_{j=0}^N\chi_jT_{\psi}u=\chi_0T_\psi u+\sum_{j=1}^N \chi_j T_{\psi}(u{\bf 1}_{\Om\cap U_j})+\sum_{j=1}^N \chi_j T_{\psi}(u{\bf 1}_{\Om\setminus U_j}).$$
It follows 
\begin{eqnarray}
\label{eqn:a1}\no{T_{\psi}u}_{L^q(\Om)}\le \no{\chi_0T_\psi u}_{L^q(\Om)}+\sum_{j=1}^N \no{\chi_j T_{\psi}(u{\bf 1}_{\Om\cap U_j})}_{L^q(\Om)}+\sum_{j=1}^N \no{\chi_j T_{\psi}(u{\bf 1}_{\Om\setminus U_j})}_{L^q(\Om)}.
\end{eqnarray}
In order to estimate $\no{\chi_0T_\psi u}_{L^q(\Om)}$ and  $\no{\chi_j T_{\psi}(u{\bf 1}_{\Om\setminus U_j})}_{L^q(\Om)}$ with $j\ge 1$, we use the continuity up to the off-diagonal boundary of the Bergman kernel. Indeed, $$K_{\Om}\in C\left(\left(\overline{\Omega}\times\overline{\Omega}\right)\setminus \left\{ \left(z,z\right):z\in b\Omega\right\} \right)$$ implies that there exists a positive constant $C$
 such that 
\[
\left|K_{\Om}\left(z,w\right)\right|\le C
\quad \T{
for all }\quad (z,w)\in \left(\bigcup_{j=1}^N \left((\T{supp}(\chi_j))\times (\Om\setminus \bar U_j)\right)\bigcup \left(\T{supp}(\chi_0))\times \Om\right)\right).\] Thus,
 for $j=1,\dots,N$, and $z\in \Om$, we have 
\begin{equation}
\left|\left(\chi_jT_{\psi} (u{\bf1}_{\Om\setminus U_j})\right)(z)\right|=\left|\int_{\Om}\chi_j(z)K(z,w){\bf1}_{\Om\setminus U_j}u(w)\psi(w)dw\right|\le  C\int_{\Om}|u(w)|dw,
\end{equation}
and hence 
\begin{eqnarray}
\label{eqn:a2}\no{\chi_jT_{\psi} (u{\bf1}_{\Om\setminus U_j})}_{L^q(\Om)}\le C\no{u}_{L^p(\Om)},\end{eqnarray}
for all $1\le p\le q\le \infty$, where $C$ is independent of $p,q$. Analogously, 
\begin{eqnarray}
\label{eqn:a3}\no{\chi_0T_{\psi} (u)}_{L^q(\Om)}\le C\no{u}_{L^p(\Om)}.\end{eqnarray}

To estimate the norm $\no{\chi_j T_{\psi}(u{\bf 1}_{\Om\cap U_j})}_{L^q(\Om)}$, $j=1,\dots,N$, we combine  the fact that
$$\no{\chi_jT_{\psi} (u{\bf1}_{U_j})}_{L^q(\Om)}\le \left(\int_{\Om\cap U_j}\left(\int_{\Om\cap U_j}|K_{\Om}(z,w)u(w)\psi(w)|dw\right)^qdz\right)^\frac{1}{q}=\no{T_{\psi,U_j}u}_{L^q(\Om\cap U_j)}$$
and Corollary~\ref{co1} to yield
\begin{eqnarray}
\label{eqn:a4}\no{\chi_jT_{\psi} (u{\bf1}_{U_j})}_{L^q(\Om)}\le C\left(\frac{p}{p-1}+q\right)^{1-\frac{1}{p}+\frac{1}{q}}\no{u}_{L^p(\Om)},\end{eqnarray}
 for given $1<p\le q<\infty$.
 From \eqref{eqn:a1}, \eqref{eqn:a2}, \eqref{eqn:a3} and \eqref{eqn:a4}, we have the desired inequality
$$\no{T_{\psi} u}_{L^q(\Om)}\le C\left(\frac{p}{p-1}+q\right)^{1-\frac{1}{p}+\frac{1}{q}}\no{u}_{L^p(\Om)},$$
for the given $1<p\le q<\infty$, provided $u\in L^p(\Om)$. 
 This proves  Part (1) in Theorem~\ref{thm:main1}.\\
 
 Similarly, the proof of Part (2) follows by
  \eqref{eqn:a1}, \eqref{eqn:a2}, \eqref{eqn:a3} (for the choice $p=1,q=\infty$) and 
  $$\no{\chi_jT_{\psi} (u{\bf1}_{U_j})}_{L^\infty(\Om)}\le \no{u}_{L^1(\Om)}$$
since 
 $$\left|\psi(w)K_{\Om}\textbf{}(z,w)\right|\le \left|K_{\Om}(z,w)\right|K_{\Om}(z,z)^{-1}\le C,\quad \T{(by \eqref{eqn:Kzw Kzz})}$$
for any $z,w\in\Om\cap U_j$ and $j=1,\dots, N$. This is complete the proof of Theorem~\ref{thm:main1}.
 \end{proof}
The following corollary is the  self $L^p$ boundedness of the Bergman projection that follows immediately from Theorem~\ref{thm:main1} and using Remark~\ref{rm:rm1} to avoid the sharp condition of the Bergman kernel. 
\begin{corollary}
		Let $\Om$ be a bounded domain in $\C^n$. Assume that the Bergman kernel  $K_{\Om}$ is of $\B$-type, that means: $(i)$  $K_\Om\in C(\overline\Om\times\overline\Om\setminus{\Delta_b})$,  and $(ii)$ for any $z\in \overline\Om$ near the boundary $b\Om$ there exists a biholomorphism $\Phi_z$ depending on $z$ but with holomorphic Jacobian uniformly nonsingular so that the Bergman kernel $K_{\Om'}$ associated to the domain $\Om':=\Phi_z(\Om)$ is of  $\B$-type at $z':=\Phi_z(z)$.
		
		Then the Bergman projection $P$ is self $L^p$ bounded for $1<p<\infty$ with the norm$$\no{P}_{L^p(\Om)\to L^p(\Om)}\le C_{\Om} \frac{p^2}{p-1},$$
			where $C_{\Om}$ is independent of $p$.
					\end{corollary}

\vspace{10pt}

\section{Sharp estimates of the Bergman-Toeplitz operators} \label{S3}
 As we have already showed in \S\ref{S2} that $T_{K^{-\alpha}}$ is bounded from $L^p(\Om)$ to $L^q(\Om)$ if $\alpha\ge \frac{1}{p}-\frac{1}{q}$. In this section we show that this gain boundedness is sharp with an additional hypothesis to the $\B$-system. \\
 
  Denote by 
  $$P_\lambda(z')=\{w'\in \C^n: |w'_j-z'_j|b_{j}(z',z')\le \lambda, \T{for all~} j=1,2\dots n\}  $$
  a $\B$-polydisc with centre $z'$ associated to the $\B$-system defined in \S\ref{S2}.

 \begin{theorem}\label{thm:main2}
 	Let $\Om$ be a bounded smooth pseudoconvex domain in $\C^n$ such that the Bergman kernel is of sharp $\B$-type. Assume further that there are {universal constants $\lambda$ and $C$} such that for any $z\in \Om$ near the boundary $b\Om$, after mapping by the biholomorphism $\Phi_z$ (in Definition~\ref{def:global B type}), the $\B$-system associated to $z':=\Phi_z(z)$ has the property:
 $P_{\lambda}(z') \subset \Om':=\Phi_z(\Om)$ and $K_{\Om'}(w',w')\le C K_{\Om'}(z',z')$, $\forall w'\in P_{\lambda}(z')$. 
 		 
 	Then, if $T_{K_{\Om}^{-\alpha}}:L^p(\Om)\to L^q(\Om)$ continuously with $1<p\le q<\infty$ then $\alpha\ge \frac{1}{p}-\frac{1}{q}$.
 \end{theorem}
 \begin{proof} We may assume $\alpha >0$. The proof is based on upper and lower estimates of $\int_\Om |K_{\Om}(w,z)|^2K_{\Om}(w,w)^{-\alpha}dw$ for $z$ approaching to the boundary. \\
 	
We first give the upper bound of $\int_\Om |K_{\Om}(w,z)|^2K_{\Om}(w,w)^{-\alpha}dw$ by using the sharp $\B$-type condition and the assumption  $T_{K_{\Om}^{-\alpha}}:L^p(\Om)\to L^q(\Om)$ continuously.
Since  $K_{\Om}(w,z)$ is holomorphic in $w\in \Om$, 
 	$$K_{\Om}(w,z)=P(K_{\Om}(\cdot,z))(w)=\int_{\Om}K_{\Om}(w,\xi)K_{\Om}(\xi,z)d\xi$$
 	and hence 
 	$$\overline{K_{\Om}(w,z)}=\overline{\int_{\Om}K_{\Om}(w,\xi)K_{\Om}(\xi,z)d\xi}=\int_{\Om}K_{\Om}(\xi,w)K_{\Om}(z,\xi)d\xi.$$  
 	Thus, we have 
 		\begin{equation}\label{eqn:Ksquare1}\begin{split}
 		\intop_{\Omega}\left|K_{\Om}\left(w,z\right)\right|^{2}K_{\Om}(w,w)^{-\alpha}dw 
 		= & \intop_{\Omega}K_{\Om}(w,z)K_{\Om}(w,w)^{-\alpha}\left(\int_{\Om}K_{\Om}(\xi,w)K_{\Om}(z,\xi)d\xi\right)dw\\
 		= & \intop_{\Om}\left(\intop_{\Om}K_{\Om}(\xi,w)K_{\Om}(w,z)K_{\Om}(w,w)^{-\alpha}dw\right)K_{\Om}(z,\xi)d\xi\\
 		=&\intop_{\Om}\Big(T_{K_\Om^{-\alpha}} (K_{\Om}(\cdot,z)))(\xi) \Big)K_{\Om}(z,\xi)d\xi\\
 		\le&\no{T_{K_\Om^{-\alpha}}(K_{\Om}(\cdot,z))}_{L^q(\Om)}\no{K_{\Om}(z,\cdot)}_{L^{q'}(\Om)}, 
  		\end{split}	\end{equation}
where the last inequality follows from H\"older's inequality and $\frac{1}{q'}+\frac{1}{q}=1$. 		
Since $T_{K_\Om^{-\alpha}}$ maps from $L^p(\Om)$ to $L^q(\Om)$ continuously then  \eqref{eqn:Ksquare1} continues as   
	\begin{equation}\label{eqn:Ksquare2}\begin{split}
 	\intop_{\Omega}\left|K_{\Om}\left(w,z\right)\right|^{2}K_{\Om}(w,w)^{-\alpha}dw \le& C\no{K_{\Om}(\cdot,z)}_{L^p(\Om)}\no{K_{\Om}(\cdot,z)}_{L^{q'}(\Om)}\\
 	\le &C \left(K_{\Om}(z,z)^{p-1}\right)^{\frac{1}{p}}\left(K_{\Om}(z,z)^{q'-1}\right)^{\frac{1}{q'}}=CK_{\Om}(z,z)^{1-\frac{1}{p}+\frac{1}{q}}.
 		\end{split}	\end{equation}
Here, the second inequality follows by using Proposition~\ref{prop:main} twice for $(a,b)=(p,0)$ and $(a,b)=(q',0)$. \\

In order to get the lower bound of $\int_\Om |K_{\Om}(w,z)|^2K_{\Om}(w,w)^{-\alpha}dw$, we first use the invariant formula \eqref{eqn:invariance} and $|\det J_\C\Phi_z(w)|\ge C$ uniformly to get 
 $$\int_\Om |K_{\Om}(w,z)|^2K_{\Om}(w,w)^{-\alpha}dw\ge C \int_{\Om'} |K_{\Om'}(w',z')|^2K_{\Om'}(w',w')^{-\alpha}dw'.$$
By the hypothesis: there exists $\lambda>0$ such that, if $w'\in P_{\lambda}(z')$ then $w'\in \Om'$ and $K_{\Om'}(w',w')\lesssim K_{\Om'}(z',z')$, it follows
\begin{equation}\begin{split}
\int_{\Om'} |K_{\Om'}(w',z')|^2K_{\Om'}(w',w')^{-\alpha}dw'\ge&  \int_{P_\lambda(z')} |K_{\Om'}(w',z')|^2K_{\Om'}(w',w')^{-\alpha}dw'\\
\gtrsim &K_{\Om'}(z',z')^{-\alpha}\int_{P_\lambda(z')} |K_{\Om'}(w',z')|^2dw'\\
\gtrsim &K_{\Om'}(z',z')^{-\alpha}K_{\Om'}(z',z')^{2}\T{Vol}(P_\lambda(z'))\\
\gtrsim &K_{\Om'}(z',z')^{1-\alpha}.
\end{split}
\end{equation}
Here the third inequality follows by  the sub-mean property and the last follows by $$\T{Vol}(P_\lambda(z'))=\pi^n\lambda^{2n}\left(\prod_{j=1}^n b_{j}(z',z')\right)^{-2}\approx (K_{\Om'}(z',z'))^{-1}.$$	
Since $K_{\Om}(z,z)=|\det J_\C\Phi_z(z)|^2K_{\Om'}(z',z')$ (by the invariant formula) and $|\det J_\C\Phi_z(z)|\approx 1$,  we get the lower bound 
\begin{eqnarray}
\label{eqn:est 3}
\intop_{\Om}|K_{\Om}(w,z)|^2K_{\Om}(w,w)^{-\alpha}dw\ge C K_{\Om}(z,z)^{1-\alpha}.
\end{eqnarray}

 Therefore, by \eqref{eqn:est 3} and \eqref{eqn:Ksquare2}, we get
 \begin{equation} 
  	K_{\Om}^{1-\frac{1}{p}+\frac{1}{q}}(z,z) \ge CK_{\Om}(z,z)^{1-\alpha},  \end{equation}
  and hence $K_{\Om}(z,z)^{\alpha-\frac{1}{p}+\frac{1}{q}}\ge C$.  Letting $z\to b\Om$, we obtain 
  $\alpha\ge \frac{1}{p}-\frac{1}{q}$ since $K_\Om(z,z)\to \infty$ as $z\to b\Om$. This proves Theorem~\ref{thm:main2}.
 \end{proof}

 \vspace{10pt}
 
 \section{Proof of Theorem~\ref{thm:main S1}} \label{S4}
 As mentioned in \S\ref{S1}, the proof of Theorem~\ref{thm:main S1} is complete if we can verify all domains listed in Theorem~\ref{thm:main S1} satisfy the hypothesis of Theorem~\ref{thm:main1} and \ref{thm:main2}. The main goal of this section is the following theorem. 
 \begin{theorem}\label{thm:main3}
 	Let $\Om$ be a bounded, pseudoconvex domain in $\C^n$ with smooth boundary. Assume that $\Om$ satisfies one of the following settings:
 	\begin{enumerate}
 		\item[(a)] $\Om$ is a strongly pseudoconvex domain;
 		\item[(b)] $\Om$ is a pseudoconvex domain of finite type and $n=2$;
 		\item[(c)] $\Om$ is a convex domain of finite type;
 		\item[(d)] $\Om$ is a decoupled domain of finite type;
 		\item[(e)] $\Om$ is a pseudoconvex domain of finite type whose Levi-form has only one degenerate eigenvalue or comparable eigenvalues.  
 	\end{enumerate} 
 	Then, the Bergman kernel $K_\Om\in C^\infty(\overline \Om\times\overline \Om\setminus \Delta_b)$ and for each point $z$ near the boundary, there is a biholomorphism $\Phi_z$ depending on $z$  but with holomorphic Jacobian uniformly nonsingular  and a $\B$-system associated to $z':=\Phi_z(z)$ such that the Bergman kernel $K_{\Om'}$ associated to $\Om':=\Phi_z(\Om)$ has the following properties:
 	\begin{enumerate}	\item 	$\left|K_{\Om'}(z',w')\right|\les \prod_{j=1}^{n}b^{2}_{j}(z',w')$, 
 		for $w'\in \Omega\cap \BB(z',c)$. Here $c$ is a universal constant.
		\item $K_{\Om'}(z',z')\approx  \prod_{j=1}^{n}b^2_{j}(z',z')$.
 		\item There exist universal constants  $\lambda$ and $C$ such that the $\B$-polysdisc $P_{\lambda}(z')\Subset \Om'$ and $K_{\Om'}(w',w')\le C K_{\Om'}(z',z')$, for any $w'\in P_{\lambda}(z')$. 
 	\end{enumerate}
\end{theorem}

\begin{proof} Inspired by the work of Kerzman \cite{Ker72}, Boas \cite{Boa87b} proves that the Bergman kernel associated to smooth, bounded, pseudoconvex domains of finite type is smooth up to off-diagonal boundary, i.e, $K_{\Om}\in C^\infty(\bar \Om\times\bar \Om\setminus \Delta_b)$. Thus, we only need to verify the local properties after the bihomomorphic mapping $\Phi_z$. \\
	
	The details of the proof of conclusion (1)-(2) can be found in paper by McNeal \cite{McN94, McN94b}, Cho \cite{Cho96, Cho02}, Catlin \cite{Cat89}. For example, the proof of the upper-bound estimate of the Bergman kernel, i.e. (1), has been established on strongly pseudoconvex domains \cite{Fef74}, pseudoconvex domains of finite type in $\C^2$ \cite[Theorem~3.1]{NaRoStWa89}, convex domains of finite type  \cite[Theorem~5.2]{McN94}, decoupled domains \cite[Theorem~2]{McN91}, and pseudoconvex domains of finite type whose Levi-form has only one degenerate eigenvalue or comparable eigenvalues \cite[Theorem~1]{Cho96} and \cite[Theorem~1.1]{Cho02}. Moreover, the sharp estimate of Bergman kernel on the diagonal 
	has been shown on strongly pseudoconvex domains \cite{Fef74}, pseudoconvex domains of finite type in $\C^2$ \cite[Theorem 2]{Cat89}, convex domains of finite type  \cite[Theorem 3.4 and Theorem 5.2]{McN94}, decoupled domains \cite[Theorem 2]{McN91}, and pseudoconvex domains of finite type whose Levi-form has only one degenerate eigenvalue or comparable eigenvalues \cite[Theorem~1]{Cho94}. The proof of (3) can be given as follows.\\
	
	By the characteristic of domains listed in this theorem, one can construct a biholomorphism $\Phi_z$ associated to a given point $z$ near the boundary of $\Om$ such that the conclusion (1) and (2) hold for the $\B$-system $\{b_j(z',\cdot)\}_{j=1}^n$ associated to $z'=\Phi_z(z)$ defined as in \S\ref{S2} by
	$$A_{jk}(z')=\sum_{k_1+k_2=k,k_1,k_2>0}\left|\frac{\di^k r_{\Om'}(z')}{{\di {z'}}_j^{k_1}\di \bar {z'}_j^{k_2}} \right|.$$
	Here $r_{\Om'}=r_{\Om}\circ\Phi_z^{-1}$. The construction of $\Phi$ hinging on the nice geometric properties on these domains also give us 
	 $$\sum_{i,j=2}^n\sum_{k=2}^m\sum_{k_1+k_2=k,k_1,k_2>0}\left|\frac{\di^k r_{\Om'}(z')}{\di {z'}_i^{k_i}\di \bar {z'}_j^{k_j}} \right||w'_i-z'_i|^{k_i}|w'_j-z'_j|^{k_j}\le C \sum_{j=2}^n\sum_{k=2}^mA_{jk}(z')|w'_j-z'_j|^k,$$ 
	 for $w'\in \BB(z',c)$.
	 Thus, by the Taylor expansion, we have 
		\begin{eqnarray}\label{eqn:Taylor}|r_{\Om'}(w')-r_{\Om'}(z')|\le C  \left(|w'_1-z'_1|+\sum_{j=2}^n\sum_{k=2}^mA_{jk}(z')|w'_j-z'_j|^k\right),
			\end{eqnarray}
for $w'\in \BB(z',c)$, for a sufficiently small $c>0$.
Thus, if we restrict $w'\in P_\lambda(z')$, then $(2|r_{\Om'}(z')|)^{-1}|w'_1-z'_1|\le \lambda$ and for each $j=2,\dots,n$; $k=2,\dots,m$,
	$$\sum_{k=2}^m\left(\frac{A_{jk}(z')}{2|r_{\Om'}(z')|}\right)^{\frac{1}{k}}|w'_j-z'_j|\le b_j(z',z')|w'_j-z'_j|\le \lambda.$$
	As a consequence,   
	\begin{eqnarray}\label{eqn:P lambda}
|w_1'-z_1'|\le 2\lambda |r_{\Om'}(z')|\quad \T{and} \quad A_{jk}(z')|w_j'-z_j'|^k\le 2\lambda^k |r_{\Om'}(z')|,
	\end{eqnarray}
	for all $w'\in P_\lambda(z')$.
	
	Combining \eqref{eqn:Taylor} and \eqref{eqn:P lambda}, we get 
	$$|r_{\Om'}(w')-r_{\Om'}(z')|\le C\lambda |r_{\Om'}(z')|, $$
	and hence $r_\Om(w)\approx r_\Om(z)$ for a sufficient small $\lambda>0$. 
This means $r_{\Om}(w)<0$. Thus,  $P_{\lambda}(z')\Subset  \Om'$ for some small $\lambda$.

Using the nice geometric properties of these domains, McNeal also proves  on convex domains of finite type \cite[Proposition 2.4]{McN94} and similarly for  other domains listed in this theorem (see \cite[Proposition 4]{McN94b} for example) that if $w'\in P_\lambda(z')$ then there exists $c$ and $\tilde\lambda$ such that $P_{c\lambda}(z')\subset P_{\tilde\lambda'}(w')$. Since $K_{\Om'}(z',z')\approx \T{Vol}(P_{\lambda}(z'))$, we have 
$$K_{\Om'}(w',w')\le CK_{\Om'}(z',z'),\quad \T{if }w'\in P_{\lambda}(z').$$ 
This completes our verification. 
\end{proof}

The following corollary follows immediately from Theorem~\ref{thm:main S1}, that is a refined result by McNeal \cite{McN94b}, Phong and Stein \cite{PhSt77} for the $L^p$ estimates of the Bergman projection $P$ and also generalises the works by Zhu \cite{Zhu06} and Zhao \cite{Zha15} for the upper bound of the norm $\no{P}_{L^p(\Om)\to L^p(\Om)}.$  
 \begin{corollary}
	Let $\Om$ be a bounded, pseudoconvex domain in $\C^n$ with smooth boundary. Assume that $\Om$ satisfies at least one of the following settings:
	\begin{enumerate}
		\item[(a)] $\Om$ is a strongly pseudoconvex domain;
		\item[(b)] $\Om$ is a pseudoconvex domain of finite type and $n=2$;
		\item[(c)] $\Om$ is a convex domain of finite type;
		\item[(d)] $\Om$ is a decoupled domain of finite type;
		\item[(e)] $\Om$ is a pseudoconvex domain of finite type whose Levi-form has only one degenerate eigenvalue or comparable eigenvalues.  
	\end{enumerate} 
	Then the Bergman projection $P$ is self $L^p$-bounded for $1<p<\infty$, with the upper bound 
		$$\no{P}_{L^p(\Om)\to L^p(\Om)}\le C \frac{p^2}{p-1}.$$
		Furthermore, the self $L^p$ boundedness of $P$ is sharp in the sense that there is no $\tilde p>p$ such that $P$ is bounded from $L^p(\Om)$ to $L^{\tilde p}(\Om)$.
\end{corollary}

\vspace{10pt}

 \section{A generalised  version of Schur's test} \label{S5}
 In this section, we introduce a generalised version of Schur's test that is an important tool of studying the $L^p\T-L^q$ estimates for Toeplitz operators.  We believe our generalised Schur's test is of some independent interest as well. 
 
 \begin{theorem}\label{thm:Schur's test} Let $\left(X,\mu\right)$, $\left(Y,\nu\right)$ be measure spaces
 	with $\sigma$-finite, positive measures; let $1<p\le q<\infty $ and $\alpha\in \R$.  Let $K:X\times Y\rightarrow\mathbb{C}$ and $\psi:Y\rightarrow\mathbb{C}$ be measurable functions. Assume that there exist positive measurable  functions $h_1$, $h_2$ on $Y$ and $g$ on $X$ such that  $$h^{-1}_1h_2\psi\in L^\infty(Y,d\nu)$$ and  the inequalities
 	\begin{eqnarray}
 	\int_{Y}\left|K(x,y)\right|^{\alpha p'}h_{1}(y)^{p'}d\nu(y)  &\leq & C_{1}g(x)^{p'},\label{1}\\
 	\int_{X}\left|K(x,y)\right|^{(1-\alpha)q}g(x)^qd\mu(x)&\leq & C_2 h_2(y)^q, \label{2}
 	\end{eqnarray}
 	hold for almost every $x\in\left(X,\mu\right)$ and $y\in\left(Y,\nu\right)$, where $\frac{1}{p}+\frac{1}{p'}=1$ and $C_1, C_2$ are positive constants. 
 	
 	Then, the Toeplitz operator $T_{\psi}$ associated to the kernel $K$ and the symbol $\psi$ defined by
 	\[
 	\left({T}_{\psi}u\right)(x):=\int_{Y}K\left(x,y\right)u\left(y\right)\psi\left(y\right)d\nu\left(y\right),
 	\] is bounded from $L^p(Y,\nu)$ into $L^q(X,\mu)$. Furthermore, 
 	$$\no{T_\psi}_{L^p(Y,\nu)\to L^q(X,\mu)}\le C_1^{\frac{p-1}{p}}C_2^{\frac{1}{q}}\no{h_1^{-1}h_2\psi}_{L^\infty(Y,\nu)}.$$ 
 \end{theorem}
 \begin{remark} We remark that Theorem~\ref{thm:Schur's test} generalises the Schur's test of Zhao in \cite[Theorem 1]{Zha15}, which is a special case of ours when $X=Y$, $\psi = 1$ and $h_1=h_2$.
 \end{remark}
 \begin{proof} 
 	
 	The proof of this theorem follows from a standard argument, as in \cite{CuMc06,Zha15}. 
 	Using  H{\"o}lder's inequality and  \eqref{1}, it follows
 	\begin{equation*}\begin{split}
 	\left|\left(T_{\psi}u\right)(x)\right|
 	=& \left|\intop_{Y}\left(\left|K\left(x,y\right)\right|^{\alpha}h_{1}(y)\right)\left(\left|K\left(x,y\right)\right|^{1-\alpha}h_{1}(y)^{-1}|u(y)||\psi(y)|\right)d\nu(y)\right|\\
 	\leq & \left(C_{1}g(x)^{p'}\right)^{\frac{1}{p'}} \left|\intop_{Y}\left|K\left(x,y\right)\right|^{\left(1-\alpha\right)p}h_{1}(y)^{-p}|u(y)|^p|\psi(y)|^pd\nu(y)\right|^{\frac{1}{p}}\\
 	\leq & C_{1}^{\frac{1}{p'}}\left(\intop_{Y}\left|K\left(x,y\right)\right|^{\left(1-\alpha\right)p}g(x)^ph_1(y)^{-p}|u(y)|^p|\psi(y)|^pd\nu(y)\right)^{\frac{1}{p}},
 	\end{split}\end{equation*}
 	and hence 
 	\begin{equation*}
 	\no{T_\psi u}^p_{L^q(X,\mu)}\le C_1^{p-1}\left(\intop_{X}\left(\intop_{Y}\left|K\left(x,y\right)\right|^{\left(1-\alpha\right)p}g(x)^ph_1(y)^{-p}|u(y)|^p|\psi(y)|^pd\nu(y)\right)^{\frac{q}{p}}d\mu(x)\right)^{\frac{p}{q}}.
 	\end{equation*}
 	Since $1<p\leq q<\infty$, we can use  Minkowski's integral inequality (see \cite[Theorem 2.4]{LiLo01} for example). From \eqref{2}, the estimate continues as 
 	\begin{equation*}\begin{split}
 	\no{T_\psi u}^p_{L^q(X,\mu)}\le& C_1^{p-1} \intop_{Y}\left(\intop_{X}\left|K\left(x,y\right)\right|^{\left(1-\alpha\right)q}g(x)^q h_1(y)^{-q}|u(y)|^q|\psi(y)|^qd\mu(x)\right)^{\frac{p}{q}}d\nu(y)\\
 	\leq & C_{1}^{p-1}\intop_{Y}\left(C_{2}h_{2}(y)^q\right)^{\frac{p}{q}} h_1(y)^{-p}|\psi(y)|^p|u(y)|^pd\nu(y)\\
 	\leq & C_{1}^{p-1}C_{2}^{\frac{p}{q}}\intop_{Y}h_{1}(y)^{-p}h_2(y)^p|\psi(y)|^p|u(y)|^pd\nu(y)\\
 	\leq&  C_{1}^{p-1}C_{2}^{\frac{p}{q}}\no{h_1^{-1}h_2\psi}^p_{L^\infty(Y,\nu)}\no{u}^p_{L^p(Y,\nu)}.
 	\end{split}\end{equation*}
 	This proves Theorem~\ref{thm:Schur's test}.
 \end{proof}

 \vspace{10pt}
 

\vspace{30pt}

\end{document}